\documentclass{amsart}
\usepackage{amsmath,amssymb,graphicx,amsthm}
\usepackage{hyperref}
\usepackage{young,youngtab}

\theoremstyle{definition}

\numberwithin{definition}{section}
\newtheorem{theorem}{Theorem}
\numberwithin{theorem}{section}

\newtheorem{lemma}[theorem]{Lemma}
\newtheorem{corollary}[theorem]{Corollary}
\newtheorem{proposition}[theorem]{Proposition}
\newtheorem{remark}[theorem]{Remark}

\newcommand{\abs}[1]{\left| #1 \right|}

\numberwithin{equation}{section}
\newcommand{\N}{\mathbb N}

\newcommand{\Z}{\mathbb Z}
\newcommand{\C}{\mathbb C}
\newcommand{\g}{\mathfrak g}
\newcommand{\h}{\mathfrak h}
\newcommand{\n}{\mathfrak n}
\newcommand{\A}{\mathcal A}

\newcommand{\sll}{\mathfrak{sl}}

\DeclareMathOperator{\Hom}{Hom}

\title{A Bethe Ansatz for Symmetric Groups}
\author{Aaron Marcus}
\email{amarcus@math.uchicago.edu}
\date{February 28, 2010}

\begin{document}

\begin{abstract}We examine the commuting elements $\theta_i=\sum_{j\neq i} \frac{s_{ij}}{z_i-z_j}$, $z_i\neq z_j$, $s_{ij}$ the transposition swapping $i$ and $j$, and we study their actions on irreducible $S_n$ representations.  By applying Schur-Weyl duality to the results of \cite{RV:QuasiKZ}, we establish a Bethe Ansatz for these operators which yields joint eigenvectors for each critical point of a master function.  By examining the asymptotics of the critical points, we establish a combinatorial description (up to monodromy) of the critical points and show that, generically, the Bethe vectors span the irreducible $S_n$ representations.
\end{abstract}

\maketitle

\section{Introduction}

Let $S_n$ be the symmetric group on $n$ elements, let $s_{ij}$ denote the transposition swapping $i$ and $j$, where $i,j\leq n$, and let $\overline{\Delta}=\{(z_1, \ldots, z_n)\in \mathbb{C}^n \mid z_i=z_j \text{ for some } i\neq j\}$ be the ``big diagonal.''  For each $z\in \mathbb{C}^n\backslash \overline{\Delta}$, we have $n$ pairwise commuting elements $\theta_{iz}\in\C[S_n]$ given by
\begin{equation*}
\theta_{iz} = \sum_{j\neq i} \frac{s_{ij}}{z_i-z_j}.
\end{equation*}
  Note that, in the future, we will suppress the dependence of $\theta_{iz}$ on $z$ by writing $\theta_i$. 
  
The operators $\theta_i$ can be viewed as a limit of the Dunkl operators $\mathcal{D}_i:\C[z_1, \ldots z_n]\to \C[z_1, \ldots z_n]$ defined by $\mathcal{D}_i(f)=\frac{\partial f}{\partial z_i}+k\sum_{j\neq i}\frac{f-s_{ij}f}{z_i-z_j}$.  It is easy to see that $\theta_i=\lim_{k\to \infty} \sum_{j\neq i}\frac{1}{z_i-z_j} - \mathcal{D}_i/k$.  In this way, the study of these operators helps to elucidate a degeneracy of the double affine Hecke algebra.  Similarly, the $\theta_i$ may be viewed as deformations of the Jucys-Murphy elements $\Theta_i\in \C[S_n]$ defined by $\Theta_i=\sum_{j<i}s_{ij}$.  Indeed, if we take the limit appropriately, $\lim z_i\theta_i=\Theta_i$.  The Jucys-Murphy elements are useful in the the representation theory of $S_n$, and provide a foothold for the analysis of the $\theta_i$.
  
Given an irreducible $S_n$ representation $W^{\lambda}$, we study the action of the $\theta_{i}$.  Since they pairwise commute, it is natural to ask if these operators act semi-simply and what their joint eigenvalues are.  Our first result is that there exists a complex rational function $S(t,z)$ (which we refer to as the \emph{master function}) and a representation valued auxiliary function $\Phi(t,z)$ such that if $t$ is a critical point for a fixed value of $z$, then $\Phi(t,z)$ is a joint eigenvector of the $\theta_i$ (which we refer to as \emph{Bethe vectors}).  Additionally, the eigenvalues for $\Phi(t,z)$ are given by $\partial_{z_i}S$.

When there are enough critical points, the Bethe vectors will span $W^{\lambda}$, and thus the master function yields both the semi-simplicity of the $\theta_i$ and computes their joint spectrum.  However, there are insufficiently many critical points for some values of $z$, and it is not directly evident that there are ever sufficiently many critical points.  To establish the existence of enough critical points, we examine the asymptotics of the critical points and construct a critical point for every standard Young tableau and show that the corresponding Bethe vectors approach the eigenvectors of the Jucys-Murphy elements.  We then argue that the existence of enough critical points asymptotically implies the existence of enough critical points generically.

The paper is organized as follows.  

In Section \ref{GH}, we first recall some basic notation for semi-simple Lie algebras and then recall some results from \cite{RV:QuasiKZ} and \cite{MTV:critical} on the Gaudin hamiltonians, including establishing a master function for the Gaudin hamiltonians and the use of the Bethe Ansatz to find their eigenvalues.  In Section \ref{SW}, we recall a construction of irreducible $S_n$ representations and a statement of Schur-Weyl duality.
In Section \ref{JM}, we review the necessary background on Jucys-Murphy elements and the combinatorics of representations of $S_n$. In Section \ref{link}, we establish a connection between the Gaudin hamiltonians and our $\theta_i$, and establish a master function associated to the eigenvalues of the $\theta_i$.  In Section \ref{asymptotics}, we examine the asymptotics of the master function and show that the critical points can be described combinatorially, which will show that, generically, the $\theta_i$ act semi-simply on Specht modules. Finally, in Section \ref{sec:proof}, we prove Theorem \ref{thm:critical} by showing that that the critical points described in Section \ref{asymptotics} can be constructed inductively.

\subsection{Acknowledgements} I would like to thank Victor Ginzburg both for suggesting this problem and for many fruitful discussions.  Additionally, I would like to thank Ian Shipman for reading an early draft of the paper and pointing out areas for improvement.

\section{Gaudin Hamiltonians} \label{GH}

Let $\mathfrak{g}$ be a semi-simple lie algebra, $\h$ a Cartan subalgebra, $R$ a set of roots, $R_+$ a choice of positive roots, and $\Delta_+$ the positive simple roots.   We then have a decomposition $\mathfrak{g}=\mathfrak{n}_-\oplus \mathfrak{h}\oplus \mathfrak{n}_+$ where $\n_+$ is the collection of positive root spaces and $\n_-$ is the collection of negative root spaces.  The killing form on $\g$ is $\langle -,-\rangle$ is defined by $\langle x,y \rangle =\operatorname{Tr}(\operatorname{ad}_x\operatorname{ad}_y)$, and up to scaling is the unique non-degenerate, invariant, symmetric bilinear form on $\g$.  The restriction of the Killing form to $\h$ is non-degenerate, and so induces an isomorphism $\varphi:\h \to \h^*$ via $\varphi(h)=\langle h, - \rangle$.  We can thus transport the Killing form to $\h^*$ by demanding this map to be an isometry.  We will use the same notation for the Killing form, it's restriction, and the transported killing form, though the usage should be clear by context.  

The Killing form gives an isomorphism $\mathfrak{g}\cong \mathfrak{g}^*$.  We then have a chain of isomorphisms $\operatorname{End}(\mathfrak{g})\cong \mathfrak{g}\otimes \mathfrak{g}^* \cong \mathfrak{g}\otimes \mathfrak{g}$, and we call the image of the identity the \emph{Casimir element}, which we denote by $\Omega$.  Note that if $\{g_i\}$ is a basis for $\mathfrak{g}$ and $\{g_i'\}$ is the corresponding dual basis, then $\Omega=\sum g_i\otimes g_i'$.  If $V_1, \ldots, V_N$ are representations of $\g$, then we can define operators $\Omega_{ij}$ by $\Omega_{ij}(v_1\otimes \cdots \otimes v_i\otimes  \cdots \otimes v_j\otimes \cdots \otimes v_N)=\sum_k v_1\otimes \cdots \otimes g_k v_i \otimes \cdots \otimes g_k'v_j \otimes\cdots \otimes v_N$.  We have that $\Omega_{ij}\in \operatorname{End}_{\C}(V_1\otimes \cdots \otimes V_N)$, and that the action of $\Omega_{ij}$ commutes with the action of $\mathfrak{g}$.  

For a fixed $z\in \C^{n}\backslash\overline{\Delta}$, we define the \emph{Gaudin Hamiltonions} by $\tilde{\theta}_i=\sum_j \frac{\Omega_{ij}}{z_i-z_j}$.  These commuting operators appear in the KZ-equation for the Gaudin model associated to a semi-simple lie algebra.

\subsection{The Bethe Ansatz}

One can find the eigenvectors and eigenvalues of $\sum_j \frac{\Omega_{ij}}{z_i-z_j}$ by applying the method of Bethe Ansatz.  In \cite{SV:HyperLie}, it is shown that if $\mathfrak{g}$ is a semi-simple Lie algebra and $V_1,\ldots V_N$ are highest weight representations, there are solutions to the system of differential equations $\partial_i f=\kappa \sum_j \frac{\Omega_{ij}}{z_i-z_j} f$ given by hypergeometric integrals, where $f$ takes values in $V=V_1\otimes V_2\otimes \cdots \otimes V_N$.  If one takes an asymptotic expansion of $f$ and take the limit as $\kappa \to \infty$, this yields an eigenvector, see \cite{RV:QuasiKZ}.

Let $\alpha_1,\ldots \alpha_r\in \h^*$ be the positive simple roots of $\g$, let $\Lambda_i$ be the highest weight of $V_i$, $\Lambda=\sum \Lambda_i$, and given $\mathbf{m}=(m_1, \ldots, m_r)\in \Z_{\geq 0}^r$, let $\Lambda_{\mathbf{m}}=\Lambda-\sum m_i \Lambda_i$.  This allows us to parameterize the weight spaces of $V$ by nonnegative integers.  Additionally, if $\beta\in \h^*$, we denote the weight space of weight $\beta$ in a representation $M$ by $M_{\beta}=\{m\in M \mid hm=\beta(h)m\}$, and we denote by $M_{\beta}^{\n}=\{m\in M_{\beta} \mid \n m=0\}$ the highest weight vectors of weight $\beta$.  We abbreviate $V_{\mathbf m}:=V_{\Lambda_{\mathbf m}}$.  

Given $\mathbf{m}$ as above, and letting $\abs{\mathbf m}=\sum m_i$, we pick coordinates on $\C^{\abs{\mathbf m}}$ as $t_1^{(1)},\ldots, t_1^{(m_1)}$, $t_2^{(1)},\ldots$, $t_r^{(1)} \ldots, t_r^{(m_r)}$.  We order the coordinates by ordering the pairs $(i,j)$ lexicographically, so that $(i,j)<(k,\ell)$ if $i<j$ or $i=j$ and $k<\ell$.  Additionally, we give coordinates on $\C^{{\mathbf{m}}+N}$ by letting $z_1, \ldots z_N$ be coordinates for $\C^N$ and viewing $\C^{{\mathbf{m}}+N}=\C^{{\mathbf{m}}}\times \C^N$.  Define the function

\begin{multline*}
\Phi_{\mathbf m}(t,z) ={\prod_{k<\ell}\left(z_k-z_{\ell}\right)^{\langle \Lambda_k, \Lambda_{\ell}\rangle}} 
                \cdot {\prod_{k,i,j}\left(z_k-t_i^{(j)}\right)^{-\langle \Lambda_k, \alpha_i\rangle}} 
                \cdot {\prod_{(i,j)<(k,\ell)}\left(t_i^{(j)}-t_k^{(\ell)}\right)^{\langle \alpha_i, \alpha_k\rangle}}
\end{multline*}

Let $S_{\mathbf m}(t,z)=\log \Phi_{\mathbf m}(t,z)$.  Explicitly,

\begin{multline}\label{eq:S}
S_{\mathbf m}(t,z) ={\sum_{k<\ell}{\langle \Lambda_k, \Lambda_{\ell}\rangle}\log\left(z_k-z_{\ell}\right)} 
                - \sum_{k,i,j}{\langle \Lambda_k, \alpha_i\rangle}\log\left(z_k-t_i^{(j)}\right)  \\
                + \sum_{(i,j)<(k,\ell)}{\langle \alpha_i, \alpha_k\rangle}\log\left(t_i^{(j)}-t_k^{(\ell)}\right)
\end{multline}

For a fixed value of $z$, say that $t$ is a critical point of $S$ if $\frac{\partial S}{\partial t_i^{(j)}}(t,z)=0$ for all $i,j$.  As we will have need to look at $S_{\mathbf{m}}$ for and its critical points for different choices of $\Lambda_i$ and $\mathbf{m}$, we will refer to critical points with a specific choice of parameters as being of \emph{weight} $\sum \Lambda_i -\sum m_i \alpha_i$.  We say that $t$ is a nondegenerate critical point if the Hessian $\left(\frac{\partial^2 S}{\partial t_{i}^{(j)}\partial t_k^{(\ell)}} \right)$ is nonsingular.  The space of nondegenerate critical points is an algebraic subset of $\C^n\backslash \overline{\Delta} \times \C^{\abs{\mathbf{m}}}\backslash \overline{\Delta}$, and the projection $p_1$ onto $\C^n \backslash \overline{\Delta}$ is quasi-finite, and moreover is \'etale when restricted to the non-degenerate critical points.  Using the inverse function theorem, we see that generically we can locally find holomorphic sections of $p_1$, which we call families of critical points.  Suppose that $t=t(z)$  is a family of nondegenerate critical points.  The following is proved in \cite{RV:QuasiKZ}.

\begin{theorem} \label{theorem:master}
Given $t(z)$ as above, there is a function $\Phi$ with $\Phi(t,z)\in V_{\mathbf m}^{\n}$ such that $\Phi(t(z),z)$ is an eigenvector for the operators $\sum_j \frac{\Omega_{ij}}{z_i-z_j}$ with corresponding eigenvalues $\frac{\partial S}{\partial z_i}(t(z),z).$
\end{theorem}

We refer to the eigenvectors generated by $\Phi$ in the theorem as \emph{Bethe vectors}.  While the Bethe vector associated to a non-degenerate critical point is nonzero, we do not a priori know that different Bethe vectors for a fixed $z$ are distinct when they have the same eigenvalues.  However, in the case of the general linear group, we have the following result from \cite{MTV:critical} which remedies this problem.

\begin{theorem}
If $\g=\sll_n$, each $V_i$ is a finite dimensional irreducible representation, and $\mathbf{m}$ is such that $(V_1 \otimes \cdots \otimes V_n)_{\mathbf{m}}^{\n}$ is nonempty and $\sum \Lambda_i - \sum m_j \alpha_j$ is dominant and integral, then, for fixed values of $z$, the critical points of $S$ yield linearly independent vectors.  Moreover, if any of the critical points are degenerate, these vectors do not span $V_{\mathbf m}^{\n}$.
\end{theorem}

\subsection{The case $\mathfrak{g}=\mathfrak{sl}_n, V_i={\C^n}$} 

Now, we specialize the results to the case of $\g=\mathfrak{sl}_n, V_i=\C^n$.  We let $\C^n$ have standard basis $v_1, \ldots v_k$, and we let $e_{ij}$ denote the matrix with a $1$ in the $(i,j)$ position and zeros elsewhere.  We view $\g$ as the space of trace zero matrices, $\h$ the subspace of diagonal matrices, $n$ the space of strictly upper triangular matrices, and $\n_-$ the space of strictly lower triangular matrices.  Let $L_i \in \h^*$ be the functional such that $L_i(a_1 e_{11} + \ldots a_n e_{nn})=a_i$.  Note that $\h^*=\C[L_1,\ldots L_n]/(\sum_i L_i=0)$.  Each $e_{ij}$ spans the root space corresponding to the root $L_i-L_j$.  Set $\alpha_i=L_i-L_{i+1}$ to be the set of simple positive roots.  When we transfer the Killing form to $\h^*$, we have that $\langle \sum a_i L_i, \sum b_j L_j \rangle = (1/2n)(\sum_i a_i b_i-(1/n)(\sum a_i)(\sum b_i))$ (see \cite{FH:RepTheory}).  Thus, 
$$\langle \alpha_i,\alpha_j\rangle=\begin{cases} 1/n \quad & i=j \\ -1/2n & \abs{i-j}=1 \\ 0 & \abs{i-j}>1 \end{cases}.$$

When we view $\mathfrak{sl}_n$ acting on $\C^n$, $v_1$ is the highest weight vector with weight $L_1$.  Since $\langle L_1,L_1 \rangle=(1/2n)(1-1/n)$ and $\langle L_1,\alpha_i\rangle=(1/2n)(\delta_{1i})$, equation $\ref{eq:S}$ becomes

\begin{multline}\label{eq:S2}
S_{\mathbf m}(t,z) =\frac{n-1}{2n^2}{\sum_{k<\ell}\log\left(z_k-z_{\ell}\right)} 
                - \frac{1}{2n}\sum_{k,j}\log\left(z_k-t_1^{(j)}\right)  \\
                + \frac{1}{n}\sum_i \sum_{j<\ell} \log\left(t_i^{(j)}-t_i^{(\ell)}\right) 
                -\frac{1}{2n} \sum_{i<n-1}\sum_{j,\ell}\log\left(t_i^{(j)}-t_{i+1}^{(\ell)}\right).
\end{multline}

We can also compute the partial derivatives.
\begin{multline}\label{eq:S2dt}
\frac{\partial S_{\mathbf m}}{\partial t_i^{(j)}}(t,z) =
                \frac{\delta_{1=i}}{2n}\sum_{k}\frac{1}{z_k-t_1^{(j)}}
                + \frac{1}{n}\sum_{j\neq\ell} \frac{1}{t_i^{(j)}-t_i^{(\ell)}} \\
                -\frac{1}{2n} \sum_{\ell}\frac{1}{t_i^{(j)}-t_{i+1}^{(\ell)}}
                -\frac{\delta_{i\neq 1}}{2n} \sum_{\ell}\frac{1}{t_{i}^{(j)}-t_{i-1}^{(\ell)}}.
\end{multline}

\begin{equation}\label{eq:S2dz}
\frac{\partial S_{\mathbf m}}{\partial z_k}(t,z)=
                   \frac{n-1}{2n^2}\sum_{k\neq\ell}\frac{1}{z_k-z_{\ell}} 
                 - \frac{1}{2n}\sum_{j}\frac{1}{z_k-t_1^{(j)}}.
\end{equation}

\section{Schur-Weyl Duality} \label{SW}

To relate the Gaudin Hamiltonians to the $\theta_i$, it is necessary to recognize the irreducible $S_N$-modules as laying in $V^{\otimes N}$.  This is given via Schur-Weyl duality, whose presentation we borrow from \cite{Fulton:YoungTableaux} and \cite{FV:PolyKZ}.

If $V$ is a finite dimensional complex vector space, $V^{\otimes N}$ has natural actions of $\operatorname{GL}(V)$ and of $S_N$ which commute with each other, and which moreover form each other's centralizers when viewed as subalgebras of $\operatorname{End}_{\C}(V^{\otimes N})$.  Moreover, when viewed as a $\operatorname{GL}(V)\times S_N$ module, $V^{\otimes N}\cong \bigoplus_{\lambda}  M^{\lambda}\otimes W^{\lambda}$, where the $M^{\lambda}$ are inequivalent, irreducible polynomial $\operatorname{GL(V)}$-modules and the $W^{\lambda}$ are inequivalent irreducible $S_N$-modules.  We therefore have an isomorphism
\begin{equation}\label{eq:homrep} W^{\lambda}=\Hom_{\operatorname{GL}(V)}(M^{\lambda},V^{\otimes N}).\end{equation}  If $\dim V\geq N$, then all the irreducible representations of $S_N$ occur.  Otherwise, the irreducible representations which occur are the \emph{Specht modules} corresponding to partitions of $N$ into at most $\dim V$ parts.

Since the $\operatorname{GL}(V)$-modules occurring in the decomposition are all polynomial, we may view them as $\operatorname{SL}(V)$ modules.  For the ease of exposition, we will then view these as $\mathfrak{sl}(V)$-modules.

Given the decomposition $\sll(V)=\n_-\oplus \h \oplus \n$ into strictly lower triangular, strictly diagonal, and strictly upper triangular matrices, the simple, finite dimensional $\sll(V)$-modules are in correspondence with dominant, integral weights.  Indeed, if $\mu$ is a weight, then up to isomorphism, there is a unique simple module $M^{\mu}$ generated by a highest weight vector of weight $\mu$, and if $\mu$ is dominant and integral, $M^{\mu}$ is finite dimensional.  Moreover, every finite dimensional simple module occurs in this manner.  If we combine the isomorphism $\Hom_{\sll(V)}(M^{\mu},N)\cong N_{\mu}^{\n}$ with $(\ref{eq:homrep})$, we see that the irreducible representations of $S_N$ correspond to highest weight spaces of $V^{\otimes N}$.

It is well known that over $\C$, simple $S_N$ modules are in correspondence with partitions of $N$.  One construction of Specht modules which works well for our purposes is as follows.  To any partition $\lambda$ of $N=\lambda_1+\cdots + \lambda_k$, with $\lambda_1 \geq \lambda_2 \geq \cdots \geq \lambda_k\geq 1$, $\lambda_i\in \N$, we associate a \emph{Young Diagram}, an array of left aligned boxes where the top row has $\lambda_1$ boxes, the second row has $\lambda_2$ boxes, etc.  Denoting the boxes of the diagram by $Y(\lambda)$, a labeling of the diagram is a bijection $T:Y(\lambda)\to \{1, \ldots, N\}$.  There is a natural action of $S_N$ on the set of labellings of $Y(\lambda)$.  

Assume that $\dim(V)=n>k$, where $k$ is the number of parts of the partition $\lambda$.  Given our basis $\{v_1, \ldots, v_n\}$ of $V$, each labeling yields an element $e_T=v_{i_1}\otimes \cdots \otimes v_{i_N}\in V^{\otimes N}$ where $i_j$ is the row of $T^{-1}(j)$.  Note that $\sigma e_T=e_{\sigma T}$ for $\sigma\in S_N$.  We remark that the $e_T$ are in correspondence with what \cite{Fulton:YoungTableaux} refers to as \emph{tabloids}, equivalence classes of labellings up to equivalence of the contents of each row.  Now, we let $v_T=\sum_{\sigma\in C(T)} (-1)^{\sigma}\sigma e_T$ where the sum is over permutations which leave the contents of each column of $T$ fixed.  The $v_T$ span an irreducible representation corresponding to $\lambda$.  Moreover, each of the $v_T$ is a highest weight vector of weight $\sum_i \lambda_i L_i$.  Therefore, we have proved the following.

\begin{theorem}  The Specht module corresponding to the partition $\lambda$ is isomorphic to$\left(V^{\otimes N}_{\sum \lambda_i L_i}\right)^{\n}$.
\end{theorem}

\section{Jucys-Murphy elements and the Young basis}\label{JM}

In this section, we briefly review another approach to the representation theory of $S_n$ over $\C$.  For more details on the material in this section, see \cite{OV:NewApproach}.

Let $[n]$ denote the $n$ element set $\{1,2, \ldots, n\}$.  From the inclusions $[1]\subset[2]\subset \cdots \subset [n]$, we have inclusions of groups $S_1\subset S_2 \subset \cdots \subset S_n$ where $S_k\subset S_n$ permutes $[k]$ and fixes all the elements of $[n]\backslash [k]$.  Using this chain of inclusions, we may study the representations of $S_n$ by making judicious use of the corresponding restriction functors, which allows us to build up an understanding inductively.  

A fundamental result which makes this perspective particularly fruitful is the \emph{branching rule} which states that if $M$ is a simple $\C[S_n]$-module, then the decomposition of $M$ into simple $\C[S_{n-1}]$ modules, no module occurs with multiplicity greater than $1$.  Phrased differently, $\operatorname{dim}_{\C}\operatorname{Hom}_{\C[S_{n-1}]}(N,\operatorname{Res}^{S_n}_{S_{n-1}}M)=0 \text{ or } \C$ whenever $N$ is a simple $S_{n-1}$ representation and $M$ is a simple $S_n$ representation.  Thus, each linear subspace of $N$ canonically determines a linear subspace of $M$.  Since there is a unique simple $\C[S_1]\cong \C$-module, and since it is one dimensional, we have that simple $\C[S_2]$-modules have a distinguished basis (up to rescaling), and by induction, so do simple $\C[S_n]$-modules.  We call such a basis a \emph{Young basis}.  Note that the elements of a Young basis of an irreducible representation $M$ are in correspondence with maximal chains of inclusions $k\subset M_2\subset M_3 \subset \cdots \subset M_n=M$ where $M_i$ is an irreducible $S_i$ representation.  Indeed, if $v$ is a basis element, then $M_i=S_i v$.  

The Young basis can be described in another way, as the joint eigenspaces of the action of the commutative algebra $GZ(n)=\sum_{i\leq n} Z(\C[S_i])\subset \C[S_n]$, the sum of the centers of the corresponding group algebras.  By (\cite{OV:NewApproach} Corollary 2.6), This algebra is generated by the \emph{Jucys-Murphy elements} $\Theta_i=\sum_{j<i} s_{ij}$.  Moreover, the simple representations of $S_n$ can be identified by their corresponding eigenvalues.

In particular, suppose that $\lambda$ is a partition of $n$ and that $T:Y(\lambda)\to [n]$ is a filling of of the Young tableau associated to $\lambda$.  We define the function $c(T,i)$ to be $y(T,i)-x(T,i)$ where $x(T,i)$ and $y(T,i)$ denote the respective row and column of $T$ which contain the entry $i$. The say the \emph{content} of $T$ is the vector $(c(T,1),c(T,2), \ldots, c(T,n))$. If $T$ is standard, that is, the entries in each row and column are increasing, then $T$ may be viewed as a chain of inclusions of Young diagrams, with the $i$th diagram being $T^{-1}([i])$.  For example, 

$$\begin{Young}
1 & 3 & 4 \cr
2 & 5 \cr
\end{Young}$$
corresponds to the inclusions
$$\yng(1)\subset \yng(1,1) \subset \yng(2,1)\subset \yng(3,1) \subset \yng (3,2)$$
where we view each inclusion as preserving the top left hand corner.

Because $\operatorname{Res}^{S_n}_{S_{n-1}}(W^{\lambda})=\bigoplus_{\lambda'}W^{\lambda'}$ where $\lambda'$ is obtained from $\lambda$ by deleting a box (see corollary 7.3.3 of \cite{Fulton:YoungTableaux}), we have that a standard labeling corresponds to an element of the Young basis for $W^{\lambda}$, and thus, each standard tableaux $T$ determines an element $w_T$ of the Young basis.

The following theorem is essentially contained in section 5 of \cite{OV:NewApproach}.

\begin{theorem}
With the notation above $\Theta_i(w_T)=c(T,i)w_T$, and the joint spectrum of the $\Theta_i$ acting on an irreducible representation $M$ determine $M$.  
\end{theorem}

\section{A master function for $S_N$} \label{link}

We wish to establish a connection between the Gaudin hamiltonians and our $\theta_i$.  To do so, we will need the following calculation.

\begin{lemma}\label{lemma:symtolie}
If $\mathfrak{g}=\mathfrak{sl}_n$, then $\Omega_{12}$ acts on $V\otimes V$ via $\Omega_{12}(v_1\otimes v_2)=\frac{1}{2n}(\frac{-v_1\otimes v_2}{n} + v_2\otimes v_1)$, so that the action of $\Omega_{12}$ coincides with that of $\frac{1}{2n}(s_{12}-1/n)$.  Therefore, on $V^{\otimes N}$, $\Omega_{ij}=\frac{s_{ij}-1/n}{2n}$.
\end{lemma}

\begin{proof}
Let $e_1, \ldots, e_n$ be the standard basis for $\C^n$.  The standard decomposition $\mathfrak{sl}_n=\mathfrak{n}_- \oplus \mathfrak{h} \oplus \mathfrak{n}$ into strictly lower triangular, strictly diagonal, and strictly upper triangular matrices, then the $e_{ij}, i<j$ is a basis of root vectors for $\n$, $e_{ij}, i>j$ is a basis of root vectors for $\n_-$, and $\mathfrak{h}=\operatorname{span}(h_i)$ where $h_i=e_{ii}-e_{n,n}$.  Moreover, with respect to the Killing form, $e_{ij}$ and $\frac{1}{2n}e_{ji}$ are dual to each other when $i\neq j$.  

Since $(e_{ij}\otimes e_{ji}) (e_k\otimes e_{\ell})=e_i\otimes e_j$ if $i=\ell$ and $j=k$, and $0$ otherwise, we have that $\left(\sum_{i\neq j} e_{ij}\otimes e_{ji}\right)(e_k\otimes e_{\ell})=e_{\ell}\otimes e_k$ if $k\neq \ell$ and $0$ if $k=\ell$.  To calculate the contribution of $\Omega_{12}$ coming from $\h$, we must calculate the duals  of the $h_i$.  If $h=\sum a_i e_ii, h'=\sum b_i e_ii$ where $\sum a_i=\sum b_i=0$, then $\langle h,h' \rangle =2n\sum a_i b_i$.  Thus, given $h_i$ as a basis for $\h$, we have a corresponding dual basis given by $h_i'=\frac{1}{2n}(e_{ii}-\frac{1}{n}\sum e_{jj})$.  Then $\sum h_i\otimes h_i'$ acts on $v_k \otimes v_{\ell}$ as multiplication by 
\begin{equation*}
\frac{1}{2n}\sum_{i=1}^{n-1} (\delta_{ki}-\delta_{kn})(\delta_{\ell i}-\frac{1}{n})=
\begin{cases}
\frac{1-1/n}{2n} \quad & k=\ell \\
\frac{-1/n}{2n}        & k\neq \ell
\end{cases}.
\end{equation*}  Combining the two contributions to $\Omega_{12}$, yields the lemma.
\end{proof}

Using the lemma, we see that, on the space $V^{\otimes n}$, $\theta_i=2n \sum_j\frac{\Omega_{ij}}{x_i-x_j}+\sum_j\frac{1}{n(x_i-x_j)}$.  Therefore, the eigenvectors of $\theta_i$ are equal to those of $\sum_j \frac{\Omega_{ij}}{x_i-x_j}$, and computing the corresponding eigenvalue for one operator automatically gives the eigenvalue for the other.

Let $V=\C^n$, and let $\lambda$ be a partition of $N$.  Then, modulo the relation in $\h^*$ that $\sum L_i=0$, we have that $\sum \lambda_i L_i = NL_1-\sum_i \alpha_i \sum_{j> i} \lambda_j$.  We may combine these observations with Theorem \ref{theorem:master} to obtain the following result.

\begin{theorem}
Let $\lambda$ be a partition of $N$, let $m_i=\sum_{j>i} \lambda_j$, and let $\theta_i$ act on the Specht modules $W^{\lambda}$.  If $S(t,z)$ is the function from $\ref{eq:S2}$, and if $t(z)$ is a non degenerate solution to the system of equations $\frac{\partial S_{\mathbf m}}{\partial t_i^{(j)}}(t,z)=0$, then there is a common eigenvector for the action of the $\theta_i$, with eigenvalues $2n\frac{\partial S_{\mathbf m}}{\partial z_i}(t(z),z)+\sum_{j\neq i}\frac{1}{n(z_i-z_j)}.$  
\end{theorem}

We remark that the critical values and eigenvalues above are the critical values and partial derivatives of 

\begin{multline}\label{eq:S3}
S'_{\mathbf m}(t,z) ={\sum_{k<\ell}\log\left(z_k-z_{\ell}\right)} 
                - \sum_{k,j}\log\left(z_k-t_1^{(j)}\right)  \\
                + 2\sum_i \sum_{j<\ell} \log\left(t_i^{(j)}-t_i^{(\ell)}\right) 
                - \sum_{i<n-1}\sum_{j,\ell}\log\left(t_i^{(j)}-t_{i+1}^{(\ell)}\right).
\end{multline}

Thus, we see that all reference to $n$ drop out, and we have a method for finding the eigenvalues of the action of $\theta_i$ which depend only on our choice of partition $\lambda$.

\section{Asymptotic behavior} \label{asymptotics}

In order to analyze the action of the $\theta_i$ and to show that generically, the Specht modules decompose into joint eigenspaces, we consider taking the limit of the operators as the $z_i$'s tend towards infinity.  

\subsection{Symmetries}

There are natural actions of $\C$, $\C^*$ and $S_n$ on $\C^{n}\backslash \overline{\Delta}$.  The $\theta_i$ are left fixed under the action $z_i\mapsto z_i+c$, they scale by $1/c$ under the action $z_i\mapsto cz_i$, and are permuted under the action $z_i \mapsto z_{\sigma(i)}$.  Therefore, there is no loss of generality in assuming that $\abs{z_1}\leq\abs{z_2}\leq \abs{z_3}\leq \cdots \leq \abs{z_n}$.  Additionally, if we wished we could fix $z_1=0, z_2=1$, though we shall not.


\subsection{Asymptotic zones}

If we wish to take a limit of the $\theta_i$ as the $z_j$ approach infinity, the limit depends on how the $z_j$ approach infinity relative to each other.  We will, in particular, consider values in the asymptotic zone $\abs{z_1}<<\abs{z_2}<<\cdots << \abs{z_n}$, that is, asymptotic solutions where $\abs{z_i/z_j}\to \infty$ when $i>j$.  In this region, $\frac{1}{z_i-z_j}\sim \frac{1}{z_i}$ if $i>j$, and so $\theta_i=\sum_{j\neq i} \frac{s_{ij}}{z_i-z_j}\sim \sum_{j<i} \frac{s_{ij}}{z_i}-\sum_{j>i} \frac{s_{ij}}{z_j}= \frac{1}{z_i}\Theta_i+o(1/z_i)$ where the $\Theta_i$ are the Jucys-Murphy elements.

To show that the $\theta_i$ act semi-simply for generic $z_i$, we will show that for each standard tableau $T$, there is a corresponding family of critical points whose asymptotic eigenvalues correspond to the eigenvalues of $w_T$.

Let $\lambda$ be a partition of $N$, and let $T$ be a standard tableau of shape $\lambda$.  We will reindex our variables $t_i^{(j)}$ with the boxes of the Young tableau as follows.  Let $Y_i(\lambda)$ denote the boxes of $Y(\lambda)$ which do not lie in the first $i$ rows.  We note that there are $\abs{Y_i(\lambda)}$ variables of the form $t_i^{(j)}$.  Thus, picking any such bijection, we see that given a box $b$ in row $k$ of tableau, we have variables $t_i^{(b)}$ whenever $i<k$.  We will abuse this notation by letting $t_i^{(j)}$ denote $t_i^{(T^{-1}(j))}$ when $j$ is an integer.  

\begin{theorem}\label{thm:critical}
Given a filled standard tableau $T$, there is a critical point of $S$ and positive rational constants $0<\beta_i^{b}<1$ such that, in the asymptotic zone $\abs{z_1}<<\abs{z_2}<<\cdots << \abs{z_n}$, $t_i^{(j)} = \beta_i^{(j)} z_{T(j)}+o(z_{T(j)})$.  At this critical point, $\frac{\partial S}{\partial z_j}= (y(T,j)-x(T,j))z_j+O(1/z_j)$.  Moreover, every critical point which is defined over the asymptotic region is of this form.
\end{theorem}

\begin{corollary}
Generically, the $\theta_i$ act semi-simply.
\end{corollary}
\begin{proof}
While this follows from the fact that the action of the Jucys-Murphy elements on $W^{\lambda}$ has joint spectrum of size $\dim(W^{\lambda})$ and the fact that this is a generic condition (see the appendix), it is useful to prove the statement by analyzing the critical points master function, as such an analysis allows us to see where semi-simplicity fails.

Let $X$ and $X'$ respectively denote the collection of critical points and degenerate critical points of $S_{\lambda}$, let $Y=\C^n\backslash \overline{\Delta}$, and let $p:X\to Y$ be the projection.  By Theorem 7.1 of \cite{MTV:critical}, the Bethe vectors corresponding to the different critical points in each fiber are linearly independent and cannot span if any of the critical points are degenerate.  We have by \ref{thm:critical} that the Bethe vectors span $W^{\lambda}$ in the asymptotic zone, and so it suffices to show that generically, we have $\dim(W^{\lambda})$ critical points.  This follows from the fact that the restriction $p\vert_{X\backslash X'}$ is \'etale and dominant.  
\end{proof}

\begin{remark}
If the projection $p:X\to Y$ is a finite morphism, then the proof can be strengthened to show that, for a fixed $z$, the Bethe vectors form a basis whenever none of the critical points are degenerate.
\end{remark}

\section{Proof of Theorem \ref{thm:critical}} \label{sec:proof}

With $S_{\mathbf{m}}$ as in \ref{eq:S}, if $N=2$, then we can eliminate the appearance of $z_1$ and $z_2$ from $\frac{\partial S_{\mathbf{m}}}{\partial t_i^{(j)}}$ by making the variable substitution $s_i^{(j)}=\frac{t_i^{(j)}-z_1}{z_1-z_2}$.  In terms of these new variables, being a $t$-critical point is equivalent to
\begin{equation}\label{eq:Tdt}
\frac{\langle \Lambda_1,\alpha_i \rangle}{s_i^{(j)}} +\frac{\langle \Lambda_2,\alpha_i \rangle}{s_i^{(j)}-1} = \sum_{(i,j)\neq (k,\ell)} \frac{\langle \alpha_i, \alpha_{k} \rangle}{s_i^{(j)}-s_k^{(\ell)}}.
\end{equation}

Solutions to this equation allow us to asymptotically build up solutions to $S$.  Let $\Lambda_1, \ldots \Lambda_{\ell}, \Lambda_{\ell+1}, \ldots \Lambda_{\ell+k}$ be weights, and let $a_i, b_i, c_i\in Z_+$, with $i\leq 1 \leq r$.  We then denote $\Lambda_1^{\prime}=\sum_{i=1}^{\ell} \Lambda_i - \sum a_i \alpha_i$ and $\Lambda_2^{\prime}=\sum_{i=\ell+1}^{N} \Lambda_i - \sum b_i \alpha_i$.  Then if $t(z)$ and $t'(z)$ are non-degenerate critical points of $\ref{eq:S}$ of weights $\Lambda_1^{\prime}$ and $\Lambda_2^{\prime}$ respectively, then they can be combined into a solution of $\ref{eq:S}$ of weight $\sum_{i=1}^{\ell+k}\Lambda_i-\sum{(a_i+b_i+c_i)\alpha_i}$ by using solutions to $\ref{eq:Tdt}$ of weight $(\Lambda_1^{\prime}+\Lambda_2^{\prime})-\sum c_i \alpha_i.$  In particularly, assume that $s=\{s_i^{(j)}\mid (1,1)\leq (i,j)\leq (r,c_i)\}$ is a solution to \ref{eq:Tdt} of weight $(\Lambda_1^{\prime}+\Lambda_2^{\prime})-\sum c_i \alpha_i.$  Then we have the following theorem.

\begin{theorem}[\cite{RV:QuasiKZ} Theorem 6.1]\label{thm:build}
In the notation above, if there is a unique critical point of \ref{eq:S} of weight $\sum_{i=1}^{\ell+k} \Lambda_i -\sum (a_i + b_i + c_i)\alpha_i$ which is asymptotically of the form $(t(z_1, \ldots, z_\ell)+O(1/{z_{\ell+1}}),z_{\ell+1}s+O(1),t'(z_{\ell+1}-z_{\ell+1},z_{\ell+2}-z_{\ell+1}, \ldots, z_{\ell+k}-z_{\ell+1})+O(1/{z_{\ell+1}}))$.
\end{theorem}

For our application of the theorem, we will also need the following calculation.

\begin{lemma} \label{lemma:Scalc}
Suppose that $a_1, \ldots, a_n \geq 0$ and $a_k>0$.  Then the system of equations
\begin{align*}
s_0&=1 \\
\frac{a_i}{s_i}&=\frac{1}{s_{i-1}-s_i}-\frac{1}{s_i-s_{i+1}} \qquad (1\leq i\leq n) \\
s_{n+1}&=0
\end{align*}
has the unique solution with 
$s_i=\prod_{j=1}^i{\left(1-\frac{1}{\sum_{k=j}^n (1+a_k)}\right)}$.  
In particular, we have that $s_1=1-\frac{1}{n+a_1+a_2+\cdots + a_n}$.
\end{lemma}

\begin{proof}  This is a straightforward computation, though we remark that the conditions on the $a_i$ ensure that there is no division by zero, so that the given solution actually exists. \end{proof}

Let $\lambda$ be a partition of $N$, let $T$ be a standard tableau of shape $\lambda$, let $T_i$ be the restriction of $T$ to $T^{-1}(\{1,2, \ldots, i\})$, and let $\lambda^{(i)}$ be the corresponding partition of $i$.  Using theorem \ref{thm:build} we may build up our critical points corresponding to a given tableau by an inductive process.  In particular, we can construct critical points $t_{T_i}$ of weight $\lambda^{(i)}_1 L_1+\lambda^{(i)}_2 L_2 + \cdots + \lambda^{(i)}_{\ell} L_{\ell}=iL_1-(\alpha_1(\sum_{j>1}\lambda^{(i)}_j)+\alpha_2(\sum_{j>2}\lambda^{(i)}_j)+\cdots+ \alpha_{\ell}(\lambda^{(i)}_{\ell}))$, corresponding asymptotically to $v_{T_i}\in W^{\lambda^{(i)}}$.  If $x(T,i+1)=k$, then the change in weight when we pass from $i$ to $i+1$ is $L_k=L_1-(\alpha_1+\cdots+\alpha_{k-1})$.

Assume we have constructed $t_{T_i}$.  Applying theorem \ref{thm:build} with $t=t_{T_i}$, $t'$ the empty critical point of weight $\Lambda_1$ with $b_j=0$ for all $j$, and $c_1=c_2=\cdots=c_{k-1}=1$, we obtain a critical point of the proper weight for each corresponding solution to equation \ref{eq:Tdt}.  To complete the construction, we must show two things:
\begin{enumerate}
\item There is a unique solution to \ref{eq:Tdt} in this case, and thus a unique critical point under consideration.
\item Asymptotically, the eigenvalues associated to this critical point are the same as those associated to $v_{T_{i+1}}$.
\end{enumerate}

For the first point, we note that if we apply equation \ref{eq:Tdt} with $\Lambda_1=\sum_j {\lambda^{(i)}_j L_j}$, $\lambda_2=L_1$, $c_1=\cdots =c_{k-1}=1$, then up to a scaling factor, we are in the situation of lemma \ref{lemma:Scalc} with $a_j=\lambda_j^{(i)}-\lambda_{j+1}^{(i)}$.  In particular, we have a unique solution $s$ and $s_1=1-\frac{1}{(k-1)+(\lambda^{(i)}_1-\lambda^{(i)}_k)}=1-\frac{1}{\lambda^{(i)}_1+(x(T,i+1)-y(T,i+1))}.$

For the second point, we must calculate the eigenvalues associated with this critical point.  The eigenvalue for $\theta_j$ is
\begin{equation}
\sum_m \frac{1}{z_j-z_{m}}-\sum_b \frac{1}{z_j-t_1^{(b)}}\sim \frac{j-1}{z_j}-\frac{c}{z_j}-\frac{d}{z_j}
\end{equation}
where $c=(j-1)-\lambda_1^{(j-1)}$ is the number of boxes in $T$ not in the first row which contain numbers less than $j$, and $d=0$ if $x(T,j)=1$ and $\frac{1}{1-\lim t_1^{(j)}/z_j}$ otherwise.  By the construction of our critical points, the asymptotic value does not change as we pass from $i$ to $i+1$, and so it suffices to calculate this value when $j=i+1$.  In this case, the asymptotic eigenvalue is $\frac{i-(i-\lambda_1^{(i)})-(\lambda^{(i)}_1+(x(T,i+1)-y(T,i+1)))}{z_{i+1}}=\frac{y(T,i+1)-x(T,i+1)}{z_{i+1}}$, as desired.

Since the critical points give rise to a linearly independent set of eigenvectors, and since we have produced $\dim W^{\lambda}$ critical points, this must account for all such points.

\appendix

\section{On genericy of semisimplicity}

In this appendix, we will recall some facts about commuting families of linear operators.

Let $V$ be a $\C$-vector space, and let $\A \subset \operatorname{End}_{\C}(V)$ be an algebra of commuting linear operators.  Given $\mu \in \operatorname{Hom}_{\C}(\A,\C)$ we define the weight space $$V_{\mu}=\bigcap_{A\in \A}\bigcup_{k\in \N} \ker\left((A-\mu(A)I)^k\right)$$.

\begin{proposition}
The weight spaces of $V$ are $\A$-invariant, i.e., $\A V_{\mu}\subset V_{\mu}$ 
\end{proposition}

\begin{proof}
If $A,B\in \A$, then $B$ commutes with $(A-\mu(A)I)^k$, and so if $(A-\mu(A)I)^k v=0$, then $(A-\mu(A)I)^k (Bv)=B((A-\mu(A)I)^k v)=B(0)=0$.  
\end{proof}

We refer to the set $\{\mu \mid V_{\mu}\neq 0\}$ as the \emph{joint spectrum} of $\A$.  By abuse of notation, if $\{A_\alpha\}$ is a set of commuting linear operators, we also use the term joint spectrum to refer to the joint spectrum of the algebra generated by the $A_{\alpha}$.  In many situations, the joint spectrum gives a lot of information about the action of $\A$.  If $V$ is finite dimensional or more generally, if the action of $\A$ is locally finite, i.e., $\A v$ is finite dimensional for every $v\in V$, then because $\C$ is algebraically closed, we must have that $V=\bigoplus_{\mu} V_{\mu}$.  Note that this can fail if we are not locally finite: the shift operator on $\C^{\N}$ has no nonzero eigenspaces.  In what follows, we will assume that $V$ has such a decomposition.

We say that the action of $\A$ on $V$ is semisimple if we can find a basis of $V$ such that, with respect to the basis, every $A\in \A$ is diagonal.  Equivalently, $V$ is the direct sum of one dimensional submodules.  Because each $V_{\mu}$ is an $\A$ submodule, we see that the action is semisimple if and only if for every $A\in \A$, we have that $\ker(A^k)=\ker(A)$.  Since $\ker(A)=V$ if and only if $A=0$, this implies that a semisimple action cannot occur if $\A$ contains any nilpotents.  Using this observation, the following proposition justifies the term semisimple action.

\begin{proposition}
Suppose that $V$ is finite dimensional.  Then $\A$ is semisimple as an algebra if and only if the action of $\A$ on $V$ is semisimple.
\end{proposition}

\begin{proof}
Since $\A$ is finitely generated and commutative, its Jacobson radical is equal to its nilradical, and so $\A$ is semisimple if and only if it has no nilpotent elements.  If $\A$ is not semisimple, then $\A$ contains nilpotents, and we see that the action of $\A$ on $V$ is not semisimple.  Conversely, assume that $\A$ is semisimple.  Then decomposing the weight spaces into simple $\A$ modules, we see no $A\in \A$ can act nilpotently.
\end{proof}

A matrix being diagonalizable is not an open condition: the identity matrix is diagonal, but no matrix of the form $\begin{pmatrix} 1 & \epsilon \\ 0 & 1 \end{pmatrix}$ with $\epsilon\neq 0$ is diagonalizable.    However, the condition that a matrix have distinct eigenvalues (which implies diagonalizability) is an open condition.  Indeed, a matrix $M$ fails to have distinct eigenvalues if and only of $C_M(\lambda)=\det(M-\lambda I)$ has repeated roots, which occurs if and only if $C_M$ and $C_M'$ have a common root.  Since this occurs exactly when the resultant $\operatorname{res}(C_M,C_M')=0$, we have a polynomial condition in the entries of $M$ for when $M$ has repeated eigenvalues, and thus the condition is Zariski open.

Similarly, semisimplicity of an action is not an open condition.  However, if every nonzero $V_{\mu}$ is one dimensional, this implies semisimplicity, and if $V$ is finite dimensional and $\A$ is finitely generated, this is a an open condition.

\begin{proposition}
Let $X\subset \operatorname{Mat}_n(\C)^k$ be the subvariety of $k$-tuples of $n\times n$ matrices which pairwise commute.  Then the subset $X'\subset X$ of $k$-tuples with joint spectrum of size $n$ is Zariski open.
\end{proposition}

\begin{proof}
Consider the map $\phi:X\times \C^k \to \operatorname{Mat}_n(\C)$ defined by $(M_1, \ldots M_k, a_1, \ldots, a_k)\mapsto \sum a_i M_i$.  If $x=(M_1, \ldots, M_k)\in X'$ with weights $\mu_1, \ldots \mu_n$ with $\mu_i(M_j)=b_{ij}$, then $\phi(x,a_1, \ldots a_k)$ has eigenvalues $c_i=\sum_j b_{ij} a_j$.  Except for a finite union of hyperplanes in $\C^k$, any $\phi(x,a_1, \ldots a_k)$ has distinct eigenvalues.  Conversely, if $y\in X\backslash X'$, then $\phi(y,a_1, \ldots, a_k)$ cannot have $n$ distinct eigenvalues.  Therefore if we let $Y\subset \operatorname{Mat_n(\C)}$ denote the matrices with $n$ distinct eigenvalues, $X'=p_X(\phi^{-1}(Y))$, which is open.
\end{proof}

\bibliographystyle{amsalpha}
\bibliography{KZ}

\end{document}